\documentclass{amsart}
\usepackage{amsmath}
\usepackage{amssymb}
\usepackage{amsthm}
\usepackage{a4wide}
\usepackage{extarrows}
\usepackage{url}
\usepackage{xcolor}
\usepackage[british]{babel}

\newtheorem{theorem}{Theorem}[section]
\newtheorem{lemma}[theorem]{Lemma}
\newtheorem{defn}[theorem]{Definition}
\newtheorem{corollary}[theorem]{Corollary}
\newtheorem{proposition}[theorem]{Proposition}
\newtheorem{remark}[theorem]{Remark}

\newtheorem{Question}[]{Question}

\newtheorem{thmx}{Theorem}

\def\aff{\mathop\mathrm{aff}\nolimits}

\def\conv{\mathop\mathrm{conv}\nolimits}

\def\cl{\mathop\mathrm{cl}\nolimits}

\def\s{\mathbb{S}}
\def\K{\mathcal{K}}

\def\R{\mathbb{R}}

\def\H{\mathcal{H}}

\def\Cc{h} 
\def\Cs{\operatorname{Chg}} 
\newcommand{\dlat}{\mathrm{d}}
\newcommand{\sfe}{{\mathbb S}^{n-1}}

\newcommand{\bigdotcup}[1][0pt]{\mathaccent\cdot{}\kern-#1\bigcup}

\numberwithin{equation}{section}

\title{Canal Classes and Cheeger Sets}
\date\today

\author{Nico Lombardi}
\address{Via C. Cattaneo, 12, 50019, Sesto Fiorentino, Firenze, Italy}
\email{nico.lombardi.1990@gmail.com}

\author{Christian Richter}
\address{Institute of Mathematics, Friedrich Schiller University of Jena, Ernst-Abbe-Platz 2, 07743 Jena, Germany}
\email{christian.richter@uni-jena.de}
 
\author{Eugenia Saor\'\i n G\'omez}
\address{ALTA Institute for Algebra, Geometry, Topology and their Applications, University of Bremen, 28334 Bremen, Germany}
\email{esaoring@uni-bremen.de}
	


\begin{document}

\begin{abstract}
Giannopoulos, Hartzoulaki and Paouris asked in \cite{GHP} whether the best ratio between volume and surface area of convex bodies sharing a given orthogonal projection onto a fixed hyperplane is attained in the limit by a cylinder over the given projection. The answer to the question is known to be negative. In this paper, we prove a characterization of the positive answer in dimension $3$, using the Cheeger set of the common projection. A partial characterization is given in higher dimensions. We also prove that certain canal classes of convex bodies provide families of convex bodies satisfying a closely related inequality for a similar ratio.
\end{abstract}

\maketitle

\section{Introduction and main results}

In this work, we investigate when the following question posed in \cite{GHP} has a positive answer. The question is known to have, in general, a negative answer \cite{FGM}. We prove a complete characterization of a positive answer in the three-dimensional case and provide partial equivalent statements in the $n$-dimensional case.

\begin{Question}[{\cite{GHP}}]\label{Q1}
    Let $u \in \sfe$ be a direction and let $C \subset u^\perp$ be an $(n-1)$-dimensional convex body, i.e., a compact and convex subset of $\R^n$. 
    Is the supremum of $\frac{\H^n(K)}{\H^{n-1}(\partial K)}$ over all convex bodies $K$ with $P_{u^\perp}(K)=C$, i.e., whose projection onto $u^\perp$ is precisely $C$, attained for right cylinders over $C\subset u^\perp$ (of infinite height)?
\end{Question} 
 Here, $\H^k(S)$ denote the $k$-dimensional (outer) Hausdorff measure of $S\subset \R^n$, $k=0,\ldots,n$, and $\sfe$ denotes the unit sphere. The dimension of a convex body is the dimension of its affine hull, i.e., the smallest affine subspace of $\R^n$ in which the convex body is contained. Further, for $u\in \sfe$, the hyperplane through the origin orthogonal to $u$ is denoted by $u^\perp$. We will denote by $\K^n$ the space of all nonempty compact and convex subsets (convex bodies) of $\R^n$, and 
by $\K^n_n$ the set of all convex bodies having interior points, i.e., full-dimensional convex bodies. We denote by $P_{u^\perp}(K)\subset u^\perp$ the orthogonal projection of $K$ onto $u^\perp$. Observe that the convex body $P_{u^\perp}(K)$ has a dimension strictly less than $n$.

The family of convex bodies having the same projection $C\subset u^\perp$ onto a given hyperplane $u^\perp$, $u\in\sfe$, is known as \textit{canal class} of $C$, see \cite[Section 7.7]{Sch}. For our purposes in this paper, we will consider only full-dimensional convex bodies in the canal class, setting thus
$\K_C:=\{ K\in \K^n_n:\ P_{u^\perp}(K)=C\}$ to be the canal class with respect to $u\in\sfe$ and $C\subset u^\perp$. With this terminology, Question~\ref{Q1} reads as follows.

\setcounter{Question}{0}
\begin{Question}[reformulated]
Let $u\in\sfe$ and let $C\subset u^\perp\subset \R^n$ be an $(n-1)$-dimensional convex body. Does it hold that
\begin{equation}\label{Question 1}
\sup\left\{ \frac{\H^n(K)}{\H^{n-1}(\partial K)}:\ K\in \K_C \right\}=\lim_{h\rightarrow \infty}\frac{\H^n(C_h)}{\H^{n-1}(\partial C_h)}
\end{equation}
where $C_h$ 
is the orthogonal cylinder in $\R^n$ over $C$ with height $h$?
\end{Question}

In the present paper, we give explicit examples in $\R^n$, $n \ge 3$, that provide negative answers to Question~\ref{Q1}. The examples happen to coincide with a construction given in \cite[Proposition 4.1]{AAFO} with a different purpose. After the question is negatively answered, we focus on the possibility to find and characterize families of convex bodies in $\R^n$ satisfying \eqref{Question 1}.

We observe that the right-hand side of \eqref{Question 1} coincides with $\frac{\H^{n-1}(C)}{\H^{n-2}(\partial C)}$ (see Section \ref{s:2} for details). This links Question \ref{Q1} intrinsically to the following inequality 
\begin{equation}\label{ineq Q2}
\frac{\H^n(K)}{\H^{n-1}(\partial K)}\leq \frac{\H^{n-1}(P_{u^\perp}(K))}{\H^{n-2}(\partial P_{u^\perp}(K))}.
\end{equation}

We stress that in Question~\ref{Q1} we denoted by $\partial K$ and $\partial C$ the relative boundaries of $K$ and $C$. Taking the latter and \eqref{ineq Q2} into account, we are forced to consider only full-dimensional convex bodies in the canal class, so that all the ratios above make sense.

We notice that in \cite{FGM}, the authors prove that inequality \eqref{ineq Q2}
fails for particular $K\in\K^n$ and $u\in\sfe$, which implies that the answer to Question \ref{Q1} is negative for certain $(n-1)$-dimensional convex bodies $C\subset u^\perp$, although this fact is not explicitly mentioned in \cite{FGM}.
Considering the intimate connection between inequality \eqref{ineq Q2} and Question~\ref{Q1}, together with 
 
\begin{equation}\label{e: connection to Cheeger}\lim_{h\rightarrow \infty}\frac{\H^n(C_h)}{\H^{n-1}(\partial C_h)}=\frac{\H^{n-1}(C)}{\H^{n-2}(\partial C)},
\end{equation}
the investigation of \eqref{Question 1} is naturally connected to the Cheeger set of the convex body $C$, i.e., the Cheeger set of the $(n-1)$-dimensional convex body, which is the projection onto $u^\perp$ of all convex bodies considered in a canal class.

The Cheeger constant of a convex body $K\in\K^n$ is defined as
\[
\Cc(K):=\sup\left\{ \frac{\H^{\dim K}(C)}{\H^{\dim K-1}(\partial C)}:\ C\subset K,\ C \text{ closed},\ \H^{\dim K-1}(\partial C)>0 \right\},
\]
where $1\leq \dim K$. 
Note that if $\dim K=k$, then $\Cc(K) \ge \frac{\H^k(K)}{\H^{k-1}(\partial K)} > 0$. We call a closed subset $C \subset K$ such that $\Cc(K)=\frac{\H^k(C)}{\H^{k-1}(\partial C)}$, if it exists, a Cheeger set of $K$.  Notice that the closedness assumption is not always assumed in the literature.
The Cheeger set of a convex body $K$ is uniquely determined and will be denoted by $\Cs(K)$, see Section \ref{s:2} for further details and references on the subject.

Within this context, we prove the following result.

\begin{thmx}\label{th: bounded by Cheeger}
Let $n\geq 2$, and let $u\in \sfe$. Let further $C\subset u^{\perp}$ be an $(n-1)$-dimensional convex body. Then the inequalities
\begin{equation}\label{ineq: bounded by Cheeger}
    \frac{\H^n(L)}{\H^{n-1}(\partial L)} < \sup\left\{ \frac{\H^n(K)}{\H^{n-1}(\partial K)}:\ K\in \K_C \right\}\leq \frac{\H^{n-1}(\Cs(C))}{\H^{n-2}(\partial\Cs(C))}
\end{equation}
hold for every $L \in \K_C$.

Moreover, for every $n \ge 3$, there exists an $(n-1)$-dimensional convex body $C \in \K^n$ such that the second inequality in \eqref{ineq: bounded by Cheeger} is strict.
\end{thmx}

If $C$ is a Cheeger set, there is equality in the right-hand side of inequality \eqref{ineq: bounded by Cheeger}. More precisely, we have the following result.

\begin{thmx}\label{thm:equality for Cheeger}
     Let $u\in \sfe$, and let $n \ge 2$. Let further $C\subset u^{\perp}$ be an $(n-1)$-dimensional convex body. If $C$ is a Cheeger set, then
     \begin{equation}\label{equality for Cheeger}
\sup\left\{ \frac{\H^n(K)}{\H^{n-1}(\partial K)}:\ K\in \K_C \right\}= \frac{\H^{n-1}(C)}{\H^{n-2}(\partial C)}.
\end{equation}
\end{thmx}

In the case $n=3$, we provide a characterization of the positive answer to Question \ref{Q1}, namely, there is a positive answer to Question \ref{Q1} if and only if the set $C\subset u^\perp$ is a Cheeger set, which happens if and only if inequality \eqref{ineq Q2}
is satisfied for this precise $u\in\sfe$.

\begin{thmx}\label{th intro: n=3} Let $u\in\sfe$, $n=3$, and let $C \in \K^n$ be an $(n-1)$-dimensional convex body, such that $C\subset u^\perp$.
    Then, the following statements are equivalent:
    \begin{enumerate}
    \item[$(i)$] Question~\ref{Q1} has an affirmative answer, i.e., \eqref{Question 1} holds for $C$.
    
    \item[$(ii)$] For any $K\in \K_C$, we have 
    \[
    \frac{\H^n(K)}{\H^{n-1}(\partial K)}\leq \frac{\H^{n-1}(C)}{\H^{n-2}(\partial C)}=\frac{\H^{n-1}(P_{u^\perp} (K))}{\H^{n-2}(\partial P_{u^\perp} (K))},
    \] 
    i.e., inequality \eqref{ineq Q2} holds true.

 \item[$(iii)$] $C$ is a Cheeger set.
    \end{enumerate}
    
\end{thmx}

The content of the latter result for $n=2$ can be already found in the literature \cite{Bonnesen, GHP}, which provides us with a positive answer of Question~\ref{Q1} for every planar convex body. Observe, that any projection of a $2$-dimensional convex body onto a hyperplane $u^\perp$ (a line in the plane) is a segment, which is a Cheeger set in $u^\perp$.

Inequality \eqref{ineq Q2} is connected to the inequality
\begin{equation}\label{ineq connected}
 \frac{\H^n(K+L)}{\H^{n-1}(\partial (K+L))}\geq \frac{\H^n(K)}{\H^{n-1}(\partial K)}+\frac{\H^{n}(L)}{\H^{n-1}(\partial L)},
\end{equation}
where the vectorial or Minkowski sum of $A,B\subset\R^n$ is given by $A+B:=\{ a+b:\ a\in A,\ b\in B \}$. Indeed, taking into consideration
\begin{equation}\label{x}
	\frac{\H^n(K+L)}{\H^{n-1}(\partial (K+L))}\geq \frac{\H^n(K)}{\H^{n-1}(\partial K)}+\frac{\H^{n}(L)}{\H^{n-1}(\partial L)}\geq \frac{\H^n(K)}{\H^{n-1}(\partial K)}  
\end{equation}
and letting $L$ be a segment, i.e., $L=[\underline{0},u]$, for some direction $u\in \sfe$, inequality \eqref{x} yields
\begin{equation}\label{ineq segment}
\frac{\H^n(K)}{\H^{n-1}(\partial K)} \le \frac{\H^n(K+[\underline{0},u])}{\H^{n-1}(\partial (K+[\underline{0},u]))}= \frac{\H^n(K)+\H^{n-1}(P_{u^\perp} (K))}{\H^{n-1}(\partial K)+\H^{n-2}(\partial P_{u^\perp} (K))},   
\end{equation}
which is exactly \eqref{ineq Q2}.

Let $u\in \sfe$ be a fixed direction. The next inequality has the same spirit as \eqref{ineq connected}, where instead of the boundary of a convex body, its projection onto $u^{\perp}$ is considered, namely
\begin{equation}\label{e: 1}
	\frac{\H^n(K+L)}{\H^{n-1}(P_{u^\perp}(K+L))}\geq  \frac{\H^n(K)}{\H^{n-1}(P_{u^\perp}(K))}+ \frac{\H^n(L)}{\H^{n-1}(P_{u^\perp}(L))}.
\end{equation}
We remark that inequality \eqref{e: 1} holds for every convex body in the case of $n=2$, coinciding with \cite[formula (7.196)]{Sch} (see also \cite[Section 50]{Bonnesen}).

Although \eqref{e: 1} fails for certain convex bodies if $n\geq 3$ (see \cite{FMMZ}), a refined version of the Brunn-Minkowski inequality allows us to prove its validity directly in an appropriate set, slightly larger than a canal class. More precisely, \eqref{e: 1} holds for a direction $u$, if the projections $P_{u^\perp}(K)$ and $P_{u^\perp}(L)$ are homothetic, i.e., there exist $\lambda > 0$ and $a \in \R^n$ such that $P_{u^\perp}(K)=\lambda P_{u^\perp}(L)+a$.

\begin{thmx}\label{Equal projec}
Let $u\in \sfe$ and let $n \ge 3$. Let $K,L \in \K^n$ be such that $P_{u^\perp}(K),P_{u^\perp}(L)\subset u^\perp$ are homothetic $(n-1)$-dimensional convex bodies. Then 
\begin{equation}\label{ineq equal projec}
\frac{\H^n(K+L)}{\H^{n-1}(P_{u^\perp}(K+L))}\geq  \frac{\H^n(K)}{\H^{n-1}(P_{u^\perp}(K))}+ \frac{\H^n(L)}{\H^{n-1}(P_{u^\perp}(L))},
\end{equation} i.e.,  inequality \eqref{e: 1} 
holds true for $K$ and $L$.

Moreover, for every $n \ge 3$, there are two convex bodies $K,L\in \K^n_n$ and a direction $u\in \sfe$, such that $\H^{n-1}(P_{u^\perp}(K))=\H^{n-1}(P_{u^\perp}(L)) \ne 0$, for which inequality \eqref{ineq equal projec} fails.

\end{thmx}

The paper is organized as follows. In Section \ref{s:2} we introduce all necessary notions, known results, and tools, which will be used throughout the proofs of the main results. In Section \ref{s:3}, we focus on negative answers to Question \ref{Q1} and we prove Theorem \ref{th: bounded by Cheeger}. In Section \ref{s:4}, we focus on the characterization results for a positive answer to Question~\ref{Q1}. Finally, in Section \ref{s:5} we deal with inequality \eqref{e: 1} and prove Theorem \ref{Equal projec}.

\section{Preliminaries and background}\label{s:2}

Our ambient space is the $n$-dimensional Euclidean space $\R^n$, endowed with the standard inner product $\langle \cdot ,\cdot \rangle$ and the associated Euclidean norm $||\cdot||$. We set $B_n$ as the $n$-dimensional unit ball, and $\sfe$ as its boundary, the Euclidean unit sphere of $\R^n$. The origin is denoted by $\underline{0}$, and for $1\leq i\leq n$, we denote by $e_i$ the $i$-th vector of the standard canonical basis of $\R^n$.

We recall that $\K^n$ is the set of all convex and compact nonempty subsets of $\R^n$, and $\K^n_n$ the set of all full-dimensional convex bodies in $\R^n$, i.e., convex bodies with nonempty interior. We note that if $K\in\K^n$ and $\dim K=n$, i.e., if $K\in\K^n_n$, then $\H^n(K)$, the $n$-dimensional (outer) Hausdorff measure of $K$, coincides with the volume or Lebesgue measure of $K$, and $\H^{n-1}(\partial K)$ is its surface area. If $K\in\K^n$ and $\dim K=n-1$, then $\H^{n-1}(K)$ is  
the $(n-1)$-dimensional volume of $K$. 
When the dimension of $K\in\K^n$ is strictly less than $n$, we will consider the relative topology with respect to $\aff(K)$, the affine hull of $K$. We will denote by $\partial K$ the boundary of $K$ relative to $\aff(K)$.

As mentioned in the introduction, a positive answer to Question \ref{Q1} turns out to be connected to the theory of Cheeger sets. 

\begin{defn}
Let $K\in \K^n_n$ be a full-dimensional convex body in $\R^n$. The Cheeger constant of $K$ is defined as
\[
h(K):=\sup\left\{ \frac{\H^n(C)}{\H^{n-1}(\partial C)}:\ C\subset K,\ C \text{ closed},\ \H^{n-1}(\partial C)>0 \right\}.
\]

A closed subset $C\subset K$, for which $h(K)=\frac{\H^n(C)}{\H^{n-1}(\partial C)}$ is called a Cheeger set of $K$. If $h(K)=\frac{\H^n(K)}{\H^{n-1}(\partial K)}$, then $K$ itself is said to be a Cheeger set.
\end{defn}

There is a vast amount of literature concerning Cheeger sets, especially in the plane, where they are, in general, well understood (cf. \cite{Kawohl}). We notice that, in the literature, the definition of the Cheeger constant of an arbitrary set $\Omega\subset \R^n$ seems to differ in the assumption $C\subset \Omega$ being closed. However, if $K\in\K^n$ is a convex body, it makes no difference, as Proposition \ref{Prop Cheeger convex body} shows.

The following proposition contains only the facts about Cheeger sets that will be needed for this paper. We refer the interested reader to the cited work and the references therein.

\begin{proposition}[{\cite[Proposition 5.2]{Parini}}]\label{Prop Cheeger convex body}
Let $\Omega\in\K^n$. There exists a unique Cheeger set $\Cs(\Omega)$ of $\Omega$. Furthermore, 
$\Cs(\Omega)$ is also convex.
\end{proposition}

We remark that, if $\Omega\subset \R^n$ is not convex, examples show that the Cheeger set may not be unique (or even exist), see \cite{Parini}.

The Brunn-Minkowski inequality is one of the most well-known and powerful inequalities within the theory of convex bodies. In essence, it provides the concavity of the $n$-th root of the volume.
\begin{theorem}[{\cite[Theorem 7.1.1]{Sch}}]\label{th: BM}
		Let $K,L\in\K^n$ be two convex bodies. The inequality 
\begin{equation}\label{ineq: BM}
			\H^n\bigl((1-\lambda) K+\lambda L\bigr)^{1/n}\geq (1-\lambda) \H^n(K)^{1/n}+\lambda \H^n(L)^{1/n}
		\end{equation}
	holds true for every $\lambda\in[0,1]$.
  Equality for some $\lambda \in (0,1)$ holds if and only if $K$ and $L$ either lie in parallel hyperplanes or are homothetic.
\end{theorem}
We refer to \cite{Gardner} for an exhaustive survey on the Brunn-Minkowski inequality, and to \cite[Chapters 7 and 9]{Sch} for a thorough analysis of the Brunn-Minkowski inequality.

When the convex bodies involved in the above inequality share a common projection onto a hyperplane $u^\perp$, for some $u\in\sfe$, the Brunn-Minkowski inequality admits a refinement, also referred to as a linear version of it. By a linear version of the inequality \eqref{ineq: BM}, we mean an inequality of the form
\begin{equation}\label{linear BM}
	\H^n\bigl((1-\lambda) K+\lambda L\bigr)\geq (1-\lambda)\H^n(K)+ \lambda\H^n(L).
\end{equation}   
Since the function $\alpha \mapsto \alpha^{1/n}$ is concave for $\alpha \ge 0$, we have
\begin{equation*}
	(1-\lambda)\H^n(K)+ \lambda\H^n(L)\geq \left((1-\lambda)\H^n(K)^{1/n}+ \lambda\H^n(L)^{1/n}\right)^n
	\end{equation*}
for every $K,L\in \K^n$ and $\lambda\in [0,1]$.
Thus, if \eqref{linear BM} holds for $K,L\in\K^n$, then \eqref{linear BM} is a refinement of the Brunn-Minkowski inequality \eqref{ineq: BM}. Taking this into account, the main task in this context consists of finding, if it exists, a suitable family of convex bodies for which \eqref{linear BM} holds. Equal $(n-1)$-dimensional 
common orthogonal projections onto a given hyperplane, or, in other words, canal classes, appear to be one useful condition defining such families. In \cite[Section 7.7]{Sch}, several results in this direction are presented. For completeness, we recall the definition of canal class given in the introduction.
	
\begin{defn}[{\cite[Section 7.7]{Sch}}]
Let $u\in \sfe$ be a unit vector and let $C \in \K^n$ be an $(n-1)$-dimensional convex body such that $C\subset u^\perp$. The canal class of $C$, with respect to $u\in\sfe$, denoted by $\K_C$, is the family of convex bodies 
\[
\K_C:=\{ K\in \K^n_n:\ P_{u^\perp}(K)=C \}.
\]
\end{defn}

As stated above, a linear refinement of the Brunn-Minkowski inequality follows, assuming that both convex bodies $K,L$ belong to the same canal class for some $u\in\sfe$ and an $(n-1)$-dimensional $C\in\K^n$, $C\subset u^\perp$. Indeed, a significantly milder assumption, namely, that the $(n-1)$-dimensional volumes of the projections of $K$ and $L$ onto $u^\perp$ coincide, also allows linear refinement.

\begin{theorem}[{\cite[Section 50]{Bonnesen}, \cite[p. 42]{AAGM}\label{th: linear BM}}]
	Let $K,L\in\K^n$ be convex bodies and let $\lambda\in[0,1]$.
 If $\H^{n-1}(P_{u^\perp}(K))=\H^{n-1}(P_{u^\perp}(L))$ for some $u\in \sfe$, then the inequality
 \begin{equation}\label{ineq: linear BM}
\H^n\bigl( (1-\lambda)K+\lambda L\bigr)\geq  (1-\lambda)\H^n(K)+\lambda\H^n(L)
\end{equation}
holds for every $\lambda\in [0,1]$.
\end{theorem}

Clearly, if $K,L\in\K_C$ for some $C\subset u^\perp\subset\R^n$, and $u\in\sfe$, Theorem \ref{th: linear BM} applies.
We refer also to \cite[Section 7.7]{Sch} and \cite{Bonnesen, Oh} for further comments and also classical aspects, and to \cite{CSGYN, DVLSG, HCYN2017, Yepes} for recent developments about linear refinements of Brunn-Minkowski type inequalities.

As a direct application of the latter theorem, we will be able to prove Theorem \ref{Equal projec}.

The following particular case of \cite[Theorem~1.1]{FGM} will be needed for proving Theorem \ref{th intro: n=3}.

\begin{lemma}[{\cite[Theorem 1.1]{FGM}}]\label{FGM}
Let $K,L \in \K^2$ be $2$-dimensional convex bodies in the plane. Then
\begin{equation}\label{ineq FGM} \nonumber   
	\frac{\H^2(K+L)}{\H^1(\partial(K+L))}\geq \frac{\H^2(K)}{\H^1(\partial K)}+\frac{\H^2(L)}{\H^1(\partial L)}\,.
\end{equation}	
\end{lemma}

Note that this result does not hold in higher dimensions with the respective ratio of $\H^n$ and $\H^{n-1}$, see also \cite[Theorem~1.1]{FGM}.

The following result from \cite{GHP}, which will be used later on, does implicitly (see \cite{FGM}) a counterexample to Question \ref{Q1}.

\begin{theorem}[{\cite[Theorem 4.3]{GHP}}]\label{th: GHP}
Let $n \ge 2$, let $K \in \K^n_n$ be an $n$-dimensional convex body, and let $u \in \sfe$. 
Then
\begin{equation}\label{e:GHP}
\frac{\H^n(K)}{\H^{n-1}(\partial K)} \le \frac{2(n-1)}{n} \frac{\H^{n-1}(P_{u^\perp} (K))}{\H^{n-2}(\partial P_{u^\perp} (K))}.
\end{equation}
\end{theorem}

\begin{remark}
In \cite{FGM}, the authors prove that inequality \eqref{e:GHP} in Theorem \ref{th: GHP} is indeed sharp. This implies directly, if $n\geq 3$, the existence of some $\epsilon >0$, of a convex body $K_\epsilon$, and of $u\in\sfe$ for which
\[
\frac{\H^n(K_\epsilon)}{\H^{n-1}(\partial K_\epsilon)} \ge (1-\epsilon)\frac{2(n-1)}{n} \frac{\H^{n-1}(P_{u^\perp} (K_\epsilon))}{\H^{n-2}(\partial P_{u^\perp} (K_\epsilon))}> \frac{\H^{n-1}(P_{u^\perp} (K_\epsilon))}{\H^{n-2}(\partial P_{u^\perp} (K_\epsilon))}.
\]
The latter is the reverse of \eqref{ineq Q2}.
\end{remark}

Next, we provide a result that will be used throughout the rest of the paper. We need the following notation. For a convex body $K\in\K^n$, a unit vector $u\in\sfe$ and $h\in \R$, we denote the section of $K$ orthogonal to $u$ at height $h$ by $[K]_h$, i.e., $[K]_h:=K\cap (h\,u+u^\perp)$.

\begin{lemma}\label{Lemma integrals}
Let $n \ge 2$, and let $u\in \sfe$. Let further $K\in \K^n_n$ be an $n$-dimensional convex body, and let $C=P_{u^{\perp}} (K)$. Then, there exists a family $(K_{\lambda})_{\lambda \ge 1}\subset \K_C$ of convex bodies in the canal class $\K_C$, such that
\begin{equation}\label{eq-lemma3.1}
\frac{\H^n(K)}{\H^{n-1}(\partial K)} < \lim_{\lambda \rightarrow \infty}\frac{\H^n(K_{\lambda})}{\H^{n-1}(\partial K_{\lambda})}=\frac{\int_{-\infty}^{+\infty}\H^{n-1}([K]_h)\dlat h}{\int_{-\infty}^{+\infty}\H^{n-2}(\partial [K]_h)\dlat h}.
\end{equation}
\end{lemma}

\begin{proof}
W.l.o.g., we may assume that $K$ is a polytope. The general claim will be obtained by approximation. Moreover, we may also assume that $u=e_n$. 

Let $F \in \K^n$ be an $(n-1)$-dimensional convex body with (not necessarily normalized) normal vector $v=(v_1,\ldots,v_n)\in\R^n$, such that $(v_1,\ldots,v_{n-1},0) \ne \underline{0}$, i.e., $\operatorname{aff}(F)$ is not parallel to $e_n^\perp$. Then,
\[
\int_{-\infty}^{+\infty} \H^{n-2}([F]_h)\dlat h=\int_{-\infty}^{+\infty} \H^{n-2}\left([P_{(v_1,\ldots,v_{n-1},0)^\perp}(F)]_h\right)\dlat h=\H^{n-1}\left(P_{(v_1,\ldots,v_{n-1},0)^\perp}(F)\right).
\]
The projection $P_{(v_1,\ldots,v_{n-1},0)^\perp}(F)$ satisfies 
\[
\H^{n-1}\left(P_{(v_1,\ldots,v_{n-1},0)^\perp}(F)\right)=\cos(\angle(v,(v_1,\ldots,v_{n-1},0))\, \H^{n-1}(F),
\]
since $v$ is the normal vector to $F$. This yields 
\[
\H^{n-1}\left(P_{(v_1,\ldots,v_{n-1},0)^\perp}(F)\right)=\sin(\angle(v,e_n))\, \H^{n-1}(F),
\]
because $(v_1,\ldots,v_{n-1},0) \perp e_n$. We obtain
\begin{equation}\label{eq-integral_facet}
\int_{-\infty}^{+\infty} \H^{n-2}([F]_h) \dlat h=\sin(\angle(v,e_n))\, \H^{n-1}(F).
\end{equation}

We use now the latter to express the left-hand quotient of \eqref{eq-lemma3.1} depending on the facets of $K$.  Let $\operatorname{lin}(\cdot)$ denote the linear hull (or span), and let $F_1,\ldots,F_k$ and $F_{b}$ and $F_t$ be the facets of $K$ with outer normal vectors $v^1,\ldots,v^k \notin \operatorname{lin}(e_n)$ and $-e_n$ and $e_n$, respectively. Observe that the bottom facet $F_b$ and the top facet $F_t$ might vanish. Then, by \eqref{eq-integral_facet},
\begin{eqnarray*}
    \frac{\H^n(K)}{\H^{n-1}(\partial K)} &=&
    \frac{\H^n(K)}{\H^{n-1}(F_b)+\H^{n-1}(F_t)+\sum_{i=1}^k \H^{n-1}(F_i)}\\
    &=& \frac{\int_{-\infty}^{+\infty} \H^{n-1}([K]_h) \dlat h}{\H^{n-1}(F_b)+\H^{n-1}(F_t)+\sum_{i=1}^k \frac{\int_{-\infty}^{+\infty}\H^{n-2}([F_i]_h)\dlat h}{\sin(\angle(v^i,e_n))}}.
\end{eqnarray*}

We use the diagonal operators $D_\lambda=\operatorname{diag}\underbrace{(1,\ldots,1,\lambda)}_{n-\text{times}}$, $\lambda \ge 1$, to define the convex bodies $K_\lambda=D_\lambda(K)$ by dilating $K$ in the direction $e_n$. Then, the facets of $K$ are transformed into the corresponding facets $F_1^\lambda,\ldots,F_k^\lambda,F_b^\lambda,F_t^\lambda$ with outer normals $v^{1,\lambda},\ldots,v^{k,\lambda},-e_n,e_n$. If $v^i=(v_1^i,\ldots,v_{n-1}^i,v_n^i)$, then
\begin{equation}\label{eq:new_normals}
v^{i,\lambda}=\left(v_1^i,\ldots,v_{n-1}^i,\frac{1}{\lambda}v_n^i\right).
\end{equation}
Observe that the latter yields
\begin{equation}\label{e: sinus}
\lim_{\lambda \to \infty} \sin(\angle(v^{i,\lambda},e_n)) = \sin\left(\frac{\pi}{2}\right)=1.
\end{equation}

Also note that
\begin{equation}\label{eq:normal}
\H^{n-1}(F_i^\lambda) < \lambda\, \H^{n-1}(F_i) \quad\text{ if } \lambda > 1,\; i \in \{1,\ldots,k\} \cup \{b,t\} \text{ and } v^i \not\perp e_n.
\end{equation}
In fact, this is clear if $v^i$ is a multiple of $e_n$ (i.e., $i \in \{b,t\}$), because then $F_i^\lambda$ is a translate of $F_i$. Otherwise we obtain
\[
0=\sin(0) < \sin(\angle(v^i,e_n)) < \sin(\angle(v^{i,\lambda},e_n)) < \sin\left(\frac{\pi}{2}\right)=1.
\]
Using \eqref{eq-integral_facet}, we get the following
\[
\begin{split}
\sin(\angle(v^{i,\lambda},e_n))\,\H^{n-1}(F_i^\lambda) &=
\int_{-\infty}^{+\infty} \H^{n-2}([F_i^\lambda]_h) \dlat h\\
&=\lambda\int_{-\infty}^{+\infty} \H^{n-2}([F_i]_h) \dlat h=\lambda\,\sin(\angle(v^i,e_n))\, \H^{n-1}(F_i)
\end{split},
\]
which yields \eqref{eq:normal}.

Next, we obtain
\begin{eqnarray}
    \frac{\H^n(K_\lambda)}{\H^{n-1}(\partial K_\lambda)} 
    &=& \frac{\int_{-\infty}^{+\infty} \H^{n-1}([K_\lambda]_h) \dlat h}{\H^{n-1}(F_b^\lambda)+\H^{n-1}(F_t^\lambda)+\sum_{i=1}^k \frac{\int_{-\infty}^{+\infty}\H^{n-2}([F^\lambda_i]_h)\dlat h}{\sin(\angle(v^{i,\lambda},e_n))}}\nonumber\\
    &=& \frac{\lambda\int_{-\infty}^{+\infty} \H^{n-1}([K]_h) \dlat h}{\H^{n-1}(F_b)+\H^{n-1}(F_t)+\sum_{i=1}^k \frac{\lambda\int_{-\infty}^{+\infty}\H^{n-2}([F_i]_h)\dlat h}{\sin(\angle(v^{i,\lambda},e_n))}}\nonumber\\
    &=& \frac{\int_{-\infty}^{+\infty} \H^{n-1}([K]_h) \dlat h}{\frac{1}{\lambda}\big(\H^{n-1}(F_b)+\H^{n-1}(F_t)\big)+\sum_{i=1}^k \frac{\int_{-\infty}^{+\infty}\H^{n-2}([F_i]_h)\dlat h}{\sin(\angle(v^{i,\lambda},e_n))}}.
    \label{eq-proof_Lem3.3}
\end{eqnarray}
Using now \eqref{e: sinus}, we get
\[
    \lim_{\lambda \to \infty}\frac{\H^n(K_\lambda)}{\H^{n-1}(\partial K_\lambda)} 
    = \frac{\int_{-\infty}^{+\infty} \H^{n-1}([K]_h) \dlat h}{\sum_{i=1}^k \int_{-\infty}^{+\infty}\H^{n-2}([F_i]_h)\dlat h}= \frac{\int_{-\infty}^{+\infty} \H^{n-1}([K]_h) \dlat h}{\int_{-\infty}^{+\infty}\H^{n-2}(\partial[K]_h)\dlat h},
\]
which is the right-hand equality in \eqref{eq-lemma3.1}.

Since $\sin(\angle(v^{i,\lambda},e_n))$ is strictly monotonically increasing in $\lambda$ by \eqref{eq:new_normals}, 
\eqref{eq-proof_Lem3.3} shows that the quotient $\frac{\H^n(K_\lambda)}{\H^{n-1}(\partial K_\lambda)}$ is strictly monotonically increasing, too. This yields the left-hand inequality in \eqref{eq-lemma3.1}, because $D_1$ is the identity, and thus, $K=K_1$.

Finally, note that the approximation of a general convex body $K$ by polytopes does not affect the strictness of the inequality in \eqref{eq-lemma3.1}. 
For the left-hand strict inequality of \eqref{eq-lemma3.1}, we prove that
\[
\frac{\H^n(K_{\lambda})}{\H^{n-1}(\partial K_\lambda)}> \frac{\H^n(K)}{\H^{n-1}(\partial K)}
\]
for every $\lambda > 1$ as follows. Note that a point $x \in \partial K$ has an outer normal vector perpendicular to $u$ (i.e., $x$ admits a tangent hyperplane which is perpendicular to $u^\perp$) if and only if $x$ belongs to the set $A=\{x \in \partial K:\ P_{u^\perp}(x) \in \partial P_{u^\perp}(K)\}$. We split $\partial K$ into the closed subset $A$ and its complement $B=\partial K \setminus A$, which is open relative to $\partial K$. By construction, $\H^{n-1}(B)\ge \H^{n-1}(P_{u^\perp}(B))=\H^{n-1}(P_{u^\perp}(K))>0$. 
Let $A_\lambda$ and $B_\lambda$ be the respective subsets of $\partial (K_\lambda)$.  Therefore, again by construction, $\H^{n-1}(A_\lambda)=\lambda\H^{n-1}(A)$ and (see \eqref{eq:normal}) $\H^{n-1}(B_\lambda)< \lambda \H^{n-1}(B)$.
Consequently,
\[
\frac{\H^{n}(K_\lambda)}{\H^{n-1}(\partial K_\lambda)}=
\frac{\lambda\H^{n}(K)}{\H^{n-1}(A_\lambda)+\H^{n-1}(B_\lambda)}>
\frac{\lambda\H^{n}(K)}{\lambda\H^{n-1}(A)+\lambda\H^{n-1}(B)}=
\frac{\H^{n}(K)}{\H^{n-1}(\partial K)},
\]
which is the sought claim.
\end{proof}

The following immediate consequence of Lemma \ref{Lemma integrals} will be used later in the paper.
\begin{remark}\label{Remark after Lemma}
Lemma~\ref{Lemma integrals} ensures that
\[
\sup\left\{ \frac{\H^n(K)}{\H^{n-1}(\partial K)}:\ K\in \K_C \right\}=\sup\left\{ \frac{\int_{-\infty}^{+\infty}\H^{n-1}([K]_h)\dlat h}{\int_{-\infty}^{+\infty}\H^{n-2}(\partial [K]_h)\dlat h}:\ K\in \K_C \right\}.
\]
\end{remark}

\section{On the negative answer to Question~\ref{Q1}}\label{s:3}
This section is devoted to an explicit negative answer to Question~\ref{Q1}.

We start providing an explicit counterexample to inequality \eqref{Question 1} in the $n$-dimensional case, $n \ge 3$, based on a convex body introduced in \cite[Section 4]{AAFO}. Given a nonempty set $A\subset \R^n$,  $\cl(A)$ stands for the closure of $A$ and $\conv(A)$ for its convex hull.

\begin{proposition}[Counterexample to inequality \eqref{Question 1}]
For every $n \ge 3$, there exists an $(n-1)$-dimensional convex body $C\subset \R^n$, and a convex body $A\in\K_C$, such that the answer to Question~\ref{Q1} is negative, i.e.,
\[
\frac{\H^n(A)}{\H^{n-1}(\partial A)}>\lim_{h\rightarrow \infty}\frac{\H^n(C_h)}{\H^{n-1}(\partial C_h)}.
\]
\end{proposition}

\begin{proof}
We consider $u=e_n$ and the $(n-1)$-dimensional unit cube $C=\sum_{i=1}^{n-1}[\underline{0},e_i]$. 

The right-hand side of \eqref{Question 1} amounts to
\[
\lim_{h\rightarrow \infty}\frac{\H^n(C_h)}{\H^{n-1}(\partial C_h)}=
\lim_{h\rightarrow \infty}\frac{\H^n(C+h[\underline{0},e_n])}{\H^{n-1}(\partial (C+h[\underline{0},e_n]))}=\lim_{h\rightarrow \infty}
\frac{h}{2(n-1)h+2}=
\frac{1}{2(n-1)}.
\]

Let $h >0$ and let $S$ be the $n$-dimensional simplex $S=\conv\left\{ \underline{0}, \frac{ e_1}{3},\ldots, \frac{e_{n-1}}{3},he_n \right\}$ with vertices $\underline{0},\frac{e_1}{3},\ldots, \frac{e_{n-1}}{3},he_n$, and corresponding opposite facets $F_0,\ldots,F_n$. We consider 

$$
A^h:=\cl((C+h[\underline{0},e_n])\setminus S).
$$

We observe that, by construction, $P_{e_n^\perp}(A^h)=C$, and thus, $A^h\in \K_C$. 

We aim to show that
\begin{equation}\label{ineq counterexample}
\frac{\H^n(A^h)}{\H^{n-1}(\partial A^h)}>\frac{1}{2(n-1)}
\end{equation}
for some suitable choice of $h=h(n)$ (which will depend on the dimension $n$).

For the volume, we have $\H^n(A^h)=\H^n(C_h)-\H^n(S)=h-\frac{h}{{3}^{n-1}n!}$. To calculate the surface area $\H^{n-1}(\partial A^h)$, we first consider the facets $F_0,\ldots, F_n$ of $S$. We observe that $F_n$ is an $(n-1)$-dimensional simplex with vertices at the origin $\underline{0}$, and at every $e_i/3$ for $1\leq i\leq n-1$. Thus, $\H^{n-1}(F_n)=\frac{1}{3^{n-1}(n-1)!}$. On the other hand, for $1\leq i\leq n-1$, the $(n-1)$-dimensional simplex $F_i$ has vertices at the origin $\underline{0}$, at every $e_j/3$, $1\leq j\le n-1$ for $j \ne i$, and at $he_n$. Thus, $\H^{n-1}(F_i)=\frac{h}{3^{n-2}(n-1)!}$. Finally, $F_0=\conv\left\{\frac{e_1}{3},\ldots,\frac{e_{n-1}}{3},he_n\right\}$ is an $(n-1)$-dimensional simplex with the $(n-2)$-dimensional bottom face $\conv\left\{\frac{e_1}{3},\ldots,\frac{e_{n-1}}{3}\right\}$ and whose height is given 
by the segment $\big[h e_n,\frac{1}{n-1}\left(\frac{e_1}{3}+\ldots+\frac{e_{n-1}}{3}\right)\big]$. Hence,
\[
\begin{split}
\H^{n-1}(F_0)&=\frac{1}{n-1}\H^{n-2}\left(\conv\left\{\frac{e_1}{3},\ldots,\frac{e_{n-1}}{3}\right\}\right)
\H^1\left(\left[he_n, \frac{1}{n-1}\left(\frac{e_1}{3}+\ldots+\frac{e_{n-1}}{3}\right)\right]\right)\\
&=\frac{1}{n-1} \frac{\sqrt{n-1}}{3^{n-2}(n-2)!}\left\|he_n-\frac{1}{n-1}\left(\frac{e_1}{3}+\ldots+\frac{e_{n-1}}{3}\right)\right\|
\\
&=\frac{1}{n-1} \frac{\sqrt{n-1}}{3^{n-2}(n-2)!} \sqrt{h^2+\frac{1}{9(n-1)}}\\&=\frac{\sqrt{n-1}}{3^{n-2}(n-1)!}h(1+\epsilon(h)),
\end{split}
\]
where $\lim_{h \to \infty} \epsilon(h)=0$.
Consequently,
\[
\begin{split}
\H^{n-1}(\partial A^h)&=\displaystyle\H^{n-1}(\partial C_h)-\sum_{i=1}^{n-1}\H^{n-1}(F_i)-\H^{n-1}(F_n)+\H^{n-1}(F_0)\\
&= \displaystyle (2(n-1)h+2)-(n-1)\frac{h}{3^{n-2}(n-1)!}-\frac{1}{3^{n-1}(n-1)!}+\frac{\sqrt{n-1}}{3^{n-2}(n-1)!}h(1+\epsilon(h))\\
&=\displaystyle 2(n-1)\left(h-\frac{h}{3^{n-1}n!}\right)\\
&\quad
+\frac{2(n-1)h
+2\cdot 3^{n-1}n!
-3n(n-1)h
-n
+3n\sqrt{n-1}h(1+\epsilon(h))}{3^{n-1}n!}
\\
&=\displaystyle 2(n-1)\H^n(A^h)+\frac{(2\cdot 3^{n-1}n!-n)-\left(3n^2-5n+2-3n\sqrt{n-1}(1+\epsilon(h))\right)h}{3^{n-1}n!}.
\end{split}
\]
Note that $\lim_{h \to \infty}\left(3n^2-5n+2-3n\sqrt{n-1}(1+\epsilon(h))\right)h=\infty$ for all $n\geq 3$. Hence, picking $h=h(n)$ large enough, we obtain $\H^{n-1}(\partial A^h)<2(n-1)\H^n(A^h)$, which yields \eqref{ineq counterexample}.
\end{proof}

In the following, we deal with the proof of Theorem \ref{th: bounded by Cheeger}. 
We begin proving the next theorem, which deals with the right-hand inequality of \eqref{ineq: bounded by Cheeger}.

\begin{theorem}\label{th: bounded by Cheeger1}
Let $u\in \sfe$ and $n\geq 2$. Let further $C\subset u^{\perp}$ be an $(n-1)$-dimensional convex body. Then,
\begin{equation}\label{ineq: bounded by Cheeger_} 
    \sup\left\{ \frac{\H^n(K)}{\H^{n-1}(\partial K)}:\ K\in \K_C \right\}\leq \frac{\H^{n-1}(\Cs(C))}{\H^{n-2}(\partial\Cs(C))}.
\end{equation}
\end{theorem}

\begin{proof}
By Lemma \ref{Lemma integrals}, it is enough to show that
\begin{equation}\label{eq-proof_ThmA_}
    \frac{\int_{-\infty}^{+\infty} \H^{n-1}([K]_h) \dlat h}{\int_{-\infty}^{+\infty}\H^{n-2}(\partial[K]_h)\dlat h} \le \frac{\H^{n-1}(\Cs(C))}{\H^{n-2}(\partial\Cs(C))}
\end{equation}
holds for every $K \in \K_C$.
In order to prove this, 
let $a=\inf\{h \in \R:\, [K]_h \ne \emptyset\}$ and $b=\sup\{h \in \R:\, [K]_h \ne \emptyset\}$. For $a < h < b$, $[K]_h$ is an $(n-1)$-dimensional convex body.
Since $[K]_h= P_{u^\perp}([K]_h)+ hu$ and 
$P_{u^\perp}([K]_h)\subset P_{u^\perp}(K)=C$, we have that $[K]_h$ is a translate of some convex body $L \subset C$.

Therefore,
\[
    \frac{\H^{n-1}([K]_h)}{\H^{n-2}(\partial[K]_h)} \le \sup\left\{\frac{\H^{n-1}(L)}{\H^{n-2}(\partial L)}:\, L \subset C,\, L \text{ closed},\, \H^{n-2}(\partial L) >0 \right\}= \frac{\H^{n-1}(\Cs(C))}{\H^{n-2}(\partial\Cs(C))}.
\]
Rearranging and integrating, we have
\[    \int_a^b \H^{n-1}([K]_h) \dlat h \le \frac{\H^{n-1}(\Cs(C))}{\H^{n-2}(\partial\Cs(C))} \int_a^b\H^{n-2}(\partial[K]_h)\dlat h,
\]
which is equivalent to \eqref{eq-proof_ThmA_}.
\end{proof}

By Theorem~\ref{th: bounded by Cheeger1} and Remark~\ref{Remark after Lemma}, we get also the following result.

\begin{proposition}\label{Cheeger sufficient} Let $u\in\sfe$, $n \ge 2$, and let $C \in \K^n$ be an $(n-1)$-dimensional convex body, such that $C\subset u^\perp$. If $C$ is a Cheeger set, then
    \[
     \frac{\int_{-\infty}^{+\infty}\H^{n-1}([K]_h)\dlat h}{\int_{-\infty}^{+\infty}\H^{n-2}(\partial [K]_h)\dlat h}\leq \frac{\H^{n-1}(C)}{\H^{n-2}(\partial C)}
    \] 
    for any $n$-dimensional convex body $K$ in the canal class of $C$, i.e., for any $K\in \K_C$.
\end{proposition}

The following lemma will allow us to prove the last part of Theorem \ref{th: bounded by Cheeger}.

\begin{lemma}\label{lem-Pyramid}
For every integer $n \ge 2$ and every real $0<\epsilon<1$, there exists an $n$-dimensional convex body $C \in \K^n_n$, such that
\[
\frac{\H^n(\Cs(C))}{\H^{n-1}(\partial \Cs(C))} > (1-\epsilon)n \frac{\H^n(C)}{\H^{n-1}(\partial C)}.
\]
\end{lemma}

\begin{proof}
Let $0<\epsilon<1$. We consider the simplex $C=C(h)=\conv\{\underline{0},h^2e_1,h^2e_2,\ldots,h^2e_{n-1},e_n\}$, where $h$ will be a \textit{large} real. Then $\H^n(C)=\frac{h^{2(n-1)}}{n!}$. The surface area $\H^{n-1}(\partial C)$ of $C$ is larger than twice the surface area $\frac{h^{2(n-1)}}{(n-1)!}$ of its bottom facet, because the area of the top facet $\conv\{h^2e_1,h^2e_2,\ldots,h^2e_{n-1},e_n\}$ decreases under its orthogonal projection onto the bottom facet, as the projection is a non-expanding linear map.

Thus,
\begin{equation}\label{eq-PfPyramid1}
\frac{\H^n(C)}{\H^{n-1}(\partial C)} < \frac{\frac{h^{2(n-1)}}{n!}}{2\frac{h^{2(n-1)}}{(n-1)!}}=\frac{1}{2n}.
\end{equation}

Now we
consider the subset $D=C \cap ([0,h]^{n-1} \times \R)$. Note that, for $h=h(\epsilon,n)$ large enough, we have  $[0,h]^{n-1} \times [0,1-\frac{\epsilon}{2}] \subset D \subset [0,h]^{n-1} \times [0,1]$.
Indeed, observe that the facet $\conv\{h^2e_1,h^2e_2,\ldots,h^2e_{n-1},e_n \}\cap ([0,h]^{n-1} \times \R)$ will approach the top facet of the cylinder $[0,h]^{n-1}\times [0,1]$, when $h$ becomes large. Hence,
\begin{equation}\label{eq-PfPyramid2}
\frac{\H^n(D)}{\H^{n-1}(\partial D)} > \frac{\H^{n}([0,h]^{n-1} \times [0,1-\frac{\epsilon}{2}])}{\H^{n-1}(\partial ([0,h]^{n-1} \times [0,1]))}=\frac{h^{n-1}(1-\frac{\epsilon}{2})}{2h^{n-1}+2(n-1)h^{n-2}}>\frac{1-\epsilon}{2},
\end{equation}
provided $h$ is sufficiently large. Since $D \subset C$, by the definition of the Cheeger set $\Cs(C)$ we get $\frac{\H^n(\Cs(C))}{\H^{n-1}(\partial \Cs(C))} \ge \frac{\H^n(D)}{\H^{n-1}(\partial D)}$. Combining this with \eqref{eq-PfPyramid2} and \eqref{eq-PfPyramid1}, we have \[
\frac{\H^n(\Cs(C))}{\H^{n-1}(\partial \Cs(C))}>\frac{1-\epsilon}{2}>(2n)\frac{1-\epsilon}{2}\frac{\H^n(C)}{\H^{n-1}(\partial C)},
\]
which completes the proof.
\end{proof}

Now we obtain the final ingredient we need for the proof of Theorem \ref{th: bounded by Cheeger}.

\begin{proposition}
\label{prop-counter_Cheeger}
For every $n \ge 3$, there exists an $(n-1)$-dimensional convex body $C \in \K^n$ such that the second inequality in \eqref{ineq: bounded by Cheeger} is strict, i.e.,
\[
\sup\left\{ \frac{\H^n(K)}{\H^{n-1}(\partial K)}:\ K\in \K_C \right\}< \frac{\H^{n-1}(\Cs(C))}{\H^{n-2}(\partial\Cs(C))}.
\]
\end{proposition}

\begin{proof}\label{thm-GHP}
Let $\epsilon > 0$ be small. Applying Lemma \ref{lem-Pyramid} to the integer $n-1\geq 2$, there is an $(n-1)$-dimensional convex body $C$ such that 
\begin{equation}\label{eq-PfPropCounter1}
\frac{\H^{n-1}(C)}{\H^{n-2}(\partial C)} < \frac{1}{(1-\epsilon)(n-1)} \frac{\H^{n-1}(\Cs(C))}{\H^{n-2}(\partial \Cs(C))}.
\end{equation}
We can assume that $C \in \K^n$ and $C \subset u^\perp$ for some $u \in \sfe$. By Theorem~\ref{th: GHP},
\begin{equation}\label{eq-PfPropCounter2}
\frac{\H^n(K)}{\H^{n-1}(\partial K)} \le \frac{2(n-1)}{n} \frac{\H^{n-1}(P_{u^\perp} (K))}{\H^{n-2}(\partial P_{u^\perp} (K))}= \frac{2(n-1)}{n} \frac{\H^{n-1}(C)}{\H^{n-2}(\partial C)}
\end{equation}
for every $K \in \K_C$. Passing to the supremum on the left-hand side of \eqref{eq-PfPropCounter2} and combining it with \eqref{eq-PfPropCounter1}, we get
\[
\sup\left\{\frac{\H^n(K)}{\H^{n-1}(\partial K)}:\ K \in \K_C\right\} \le
\frac{2(n-1)}{n} \frac{\H^{n-1}(C)}{\H^{n-2}(\partial C)} < \frac{2}{(1-\epsilon) n} \frac{\H^{n-1}(\Cs(C))}{\H^{n-2}(\partial \Cs(C))}.
\]
This yields the claimed strict inequality in \eqref{ineq: bounded by Cheeger}, provided $\epsilon$ is chosen so small that $\frac{2}{(1-\epsilon) n} < 1$.
\end{proof}

\begin{proof}[Proof of Theorem \ref{th: bounded by Cheeger}]
The strict left-hand inequality in \eqref{ineq: bounded by Cheeger} is obtained from inequality \eqref{eq-lemma3.1} in Lemma~\ref{Lemma integrals}, and from Remark~\ref{Remark after Lemma}. We observe that the family of convex bodies $(K_{\lambda})_{\lambda\geq 1}$ provided by Lemma \ref{Lemma integrals} satisfies $K_\lambda \in \K_C$ for all $\lambda\geq 1$. Theorem \ref{th: bounded by Cheeger1} proves the right-hand inequality of \eqref{ineq: bounded by Cheeger}. Finally, Proposition~\ref{prop-counter_Cheeger} provides a proof of the remainder of Theorem~\ref{th: bounded by Cheeger}.
\end{proof}

\section{Aspects of a positive answer to Question \ref{Q1}}\label{s:4}
In this section, we investigate the framework of a possible positive answer to Question \ref{Q1}. Our first observation is the following corollary of Theorem \ref{th: bounded by Cheeger}, which provides a sufficient condition for a positive answer to \eqref{ineq Q2}.

\begin{corollary}\label{th: corollary1} Let $n\geq 2$, and let $u\in \sfe$ be a direction. Let further $C\in \K^n$ be an $(n-1)$-dimensional Cheeger set, such that $C\subset u^{\perp}$. Then, the inequality
    \begin{equation}\label{ineq corollary refinement} 
        \frac{\H^n(K)}{\H^{n-1}(\partial K)}<  \frac{\H^{n-1}(P_{u^{\perp}}(K))}{\H^{n-2}(\partial P_{u^{\perp}}(K))}
    \end{equation}
    holds for every $K\in \K_C$.
\end{corollary}
Of course, in this case, in which $C$ is a Cheeger set, we also have that 
\[
\sup\left\{ \frac{\H^n(K)}{\H^{n-1}(\partial K)}:\ K\in \K_C \right\}\leq \frac{\H^{n-1}(P_{u^{\perp}}(K))}{\H^{n-2}(\partial P_{u^{\perp}}(K))}. 
\]
Taking this into consideration, we can provide a proof of Theorem \ref{thm:equality for Cheeger}.
\begin{proof}[Proof of Theorem \ref{thm:equality for Cheeger}]
Let $C\subset u^\perp$ be an $(n-1)$-dimensional Cheeger set, $C\in\K^n$. Then, $P_{u^\perp}(K)=C$ is a Cheeger set too. We recall that $C_h$ denotes the cylinder over $C$. Now, Theorem~\ref{th: bounded by Cheeger} together with \eqref{e: connection to Cheeger}, i.e., 
\[
\lim_{h\rightarrow \infty}\frac{\H^n(C_h)}{\H_{n-1}(\partial C_h)}=\frac{\H^{n-1}(C)}{\H^{n-2}(\partial C)},
\]
ensure that
\[
\sup\left\{ \frac{\H^n(K)}{\H^{n-1}(\partial K)}:\ K\in \K_C \right\}= \frac{\H^{n-1}(C)}{\H^{n-2}(\partial C)},
\]
as $C_h\in \K_C$, for every $h>0$.
\end{proof}

Theorem \ref{th: bounded by Cheeger} also allows us to state the following result, whose proof is an immediate consequence of both inequalities proven in Theorem \ref{th: bounded by Cheeger}.

\begin{corollary}\label{th: corollary2}
    Let $K \in \K^n_n$ be an $n$-dimensional convex body, $n \ge 2$. Then,
    \begin{equation}
        \frac{\H^n(K)}{\H^{n-1}(\partial K)}<  \min\left\{\frac{\H^{n-1}(\Cs(P_{u^{\perp}}(K)))}{\H^{n-2}(\partial\Cs( P_{u^{\perp}}(K)))}:\ u \in \sfe\right\}.
    \end{equation}
\end{corollary}

\subsection{On a (partial) positive answer to Question \ref{Q1} in arbitrary dimension}
Now we obtain some equivalent conditions to a positive answer to Question \ref{Q1}. Our main result in this direction is the following theorem.

\begin{theorem}\label{Characterization theorem n}
    Let $u\in\sfe$, $n \ge 2$, and let $C \in \K^n$ be an $(n-1)$-dimensional convex body, such that $C\subset u^\perp$.
    Then, the following statements are equivalent: 
    \begin{enumerate}
    \item[$(i)$] Question~\ref{Q1} has an affirmative answer, i.e., \eqref{Question 1} holds for $C$.
    
    \item[$(ii)$] For any $K\in \K_C$, we have 
    \[
    \frac{\H^n(K)}{\H^{n-1}(\partial K)}\leq \frac{\H^{n-1}(C)}{\H^{n-2}(\partial C)}=\frac{\H^{n-1}(P_{u^{\perp}}(K))}{\H^{n-2}(\partial P_{u^{\perp}}(K))}.
    \] 

    \item[$(iii)$] For any $K\in \K_C$ and $h \in \R$, let $[K]_h:=K \cap (u^{\perp}+hu)$, as above. Then,
    \[
     \frac{\int_{-\infty}^{+\infty}\H^{n-1}([K]_h)\dlat h}{\int_{-\infty}^{+\infty}\H^{n-2}(\partial [K]_h)\dlat h}\leq \frac{\H^{n-1}(C)}{\H^{n-2}(\partial C)}.
    \] 
    \end{enumerate}
    \end{theorem}

\begin{proof}
We start by proving the equivalence $(i)\ \Leftrightarrow\ (ii)$. 

We observe first, that, for the cylinder $C_h$ of height $h$ over $C$, we have $\H^n(C_h)=h\H^{n-1}(C)$ and $\H^{n-1}(\partial C_h)=2\H^{n-1}(C)+h\H^{n-2}(\partial C)$. This means, as we had implicitly in \eqref{e: connection to Cheeger}, that 
    \[
    \lim_{h\rightarrow \infty}\frac{\H^n(C_h)}{\H^{n-1}(\partial C_h)}=\frac{\H^{n-1}(C)}{\H^{n-2}(\partial C)}.
    \]    
Hence, if $(i)$ holds, then we have, for 
    any $K\in \K_C$, that 
    \[
    \frac{\H^n(K)}{\H^{n-1}(\partial K)}\leq \lim_{h\rightarrow \infty}\frac{\H^n(C_h)}{\H^{n-1}(\partial C_h)}=\frac{\H^{n-1}(C)}{\H^{n-2}(\partial C)}.
    \]   
Conversely, if $(ii)$ holds, then 
    \[
    \sup\left\{ \frac{\H^n(K)}{\H^{n-1}(\partial K)}:\ K\in \K_C \right\}\leq \frac{\H^{n-1}(C)}{\H^{n-2}(\partial C)}=\lim_{h\rightarrow \infty}\frac{\H^n(C_h)}{\H^{n-1}(\partial C_h)}.
    \]
Since $C_h\in \K_C$ for any $h>0$, also the reverse inequality holds.
    
Finally, the equivalence $(ii)\ \Leftrightarrow\ (iii)$ is obtained from Remark \ref{Remark after Lemma}.
\end{proof}

\begin{remark}
    The equivalence $(i)\ \Leftrightarrow\ (ii)$ in Theorem \ref{Characterization theorem n} has been implicitly stated in \cite{GHP}. 
\end{remark}

The next corollary provides us with a first use of Cheeger sets within the context of a positive answer to Question~\ref{Q1}.

\begin{corollary}\label{Sufficient Cheeger}
Let $u\in\sfe$, $n \ge 2$, and let $C \in \K^n$ be an $(n-1)$-dimensional convex body, such that $C\subset u^\perp$. If $C$ is a Cheeger set, then the three statements of Theorem \ref{Characterization theorem n} hold.
\end{corollary}

\begin{proof}
Using Theorem \ref{Characterization theorem n}, it is enough to prove any of the items in it. Proposition \ref{Cheeger sufficient} ensures that Theorem \ref{Characterization theorem n} (iii) holds true, as $C$ is a Cheeger set, which finishes the proof.
\end{proof}

In the next subsection, we will prove that, for $n=3$, the condition of being a Cheeger set is indeed necessary for any, and thus, all, of the statements in Theorem
\ref{Characterization theorem n}.

\subsection{Characterization of a positive answer for $n=3$:  Theorem \ref{th intro: n=3}}

In this part of the section, we provide necessary and sufficient conditions in dimension $n=3$ for a positive answer to Question \ref{Q1}. We characterize all $2$-dimensional convex bodies $C\subset\R^3$, for which the answer to Question~\ref{Q1} is affirmative. More precisely we prove that the condition of being a Cheeger set, proven to be sufficient in Corollary \ref{th: corollary1}, is also necessary in this ``low dimensional'' case. This allows us to state a fourth statement equivalent to the three ones in Theorem \ref{Characterization theorem n}, if $n=3$. We will shortly discuss the case $n=2$
at the end of the subsection.

\begin{theorem}[Theorem \ref{th intro: n=3}]\label{Characterization theorem n=3}
    Let $u\in\mathbb{S}^2$. Let further $C \in \K^3$ be a $2$-dimensional convex body, such that $C\subset u^\perp$.
    Then, the following four statements are equivalent:
    \begin{enumerate}
    \item[$(i)$] Question~\ref{Q1} has an affirmative answer, i.e., \eqref{Question 1} holds for $C$.
    
    \item[$(ii)$] For any $K\in \K_C$, we have 
    \[
    \frac{\H^3(K)}{\H^{2}(\partial K)}\leq \frac{\H^{2}(C)}{\H^{1}(\partial C)}=\frac{\H^2(P_{u^{\perp}}(K))}{\H^1(\partial P_{u^{\perp}}(K))}.
    \] 

    \item[$(iii)$] For any $K\in \K_C$,
    \[
     \frac{\int_{-\infty}^{+\infty}\H^{2}([K]_h)\dlat h}{\int_{-\infty}^{+\infty}\H^{1}(\partial [K]_h)\dlat h}\leq \frac{\H^{2}(C)}{\H^{1}(\partial C)}.
    \] 
    \item[$(iv)$] $C$ is a Cheeger set.
    \end{enumerate}
    
\end{theorem}

\begin{proof}
From Theorem \ref{Characterization theorem n}, we directly obtain the equivalences of $(i),(ii)$ and $(iii)$. Further, Proposition \ref{Cheeger sufficient} yields that $(iv)$ implies $(iii)$. Thus, it is enough to prove that $(iii)$ implies $(iv)$.  We proceed by contradiction.

 We start assuming that $(iv)$ fails. Then $C$ is not a Cheeger set, i.e., there is a $2$-dimensional convex body $B\subset C$, such that 
    \begin{equation}\label{contradiction1}
    \frac{\H^2(B)}{\H^1(\partial B)}>\frac{\H^2(C)}{\H^1(\partial C)}.
    \end{equation}
    We consider now the convex body $K:=\conv(C \cup (B+u))=\bigcup_{h=0}^1 (1-h)C+h(B+u)\in \K_C$. Note that $[K]_h=((1-h)C+hB)+hu$, so that the map $h \mapsto [K]_h$ is continuous on $[0,1]$ in the Hausdorff metric. We consider the two continuous functions $f(h):=\H^2([K]_h)$ and $g(h):=\H^1(\partial [K]_h)$ defined on $[0,1]$.
    Lemma~\ref{FGM} yields
    \[
    \begin{split}
    \frac{f(h)}{g(h)}&=\frac{\H^2((1-h)C+hB)}{\H^1(\partial((1-h)C+hB))}
    \geq \frac{\H^2((1-h)C)}{\H^1(\partial ((1-h)C))}+\frac{\H^2(hB)}{\H^1(\partial (hB))}\\
    &=(1-h)\frac{\H^2(C)}{\H^1(\partial C)}+h\frac{\H^2(B)}{\H^1(\partial B)}.
    \end{split}
    \]
    By \eqref{contradiction1}, we have
    \[
\frac{f(h)}{g(h)}>\frac{\H^2(C)}{\H^1( \partial C)}
    \]
    for any $h\in (0,1]$, which implies 
    \[
\frac{\int_{-\infty}^{+\infty}\H^{2}([K]_h)\dlat h}{\int_{-\infty}^{+\infty}\H^{1}(\partial [K]_h)\dlat h}
    =\frac{\int_0^1 f(h)\dlat h}{\int_0^1 g(h)\dlat h}>\frac{\H^2(C)}{\H^1(\partial C)},
    \]
    contrary to $(iii)$.
\end{proof}

The following direct consequence of the equivalence $(ii)\ \Leftrightarrow (iv)$ in Theorem~\ref{Characterization theorem n=3} sharpens Corollary~\ref{th: corollary1} in the $3$-dimensional case. 
Moreover, it gives \textit{the appropriate} answer to Question~\ref{Q1} in the $3$-dimensional case. 
\begin{corollary}\label{cor: dim 3 Cheeger charact}
Let $u\in \s^2$, and let $C\in \K^3$ be a $2$-dimensional convex body, such that $C\subset u^{\perp}$. Then, the inequality
    \[ 
        \frac{\H^3(K)}{\H^2(\partial K)}\leq  \frac{\H^2(P_{u^{\perp}}(K))}{\H^{1}(\partial P_{u^{\perp}}(K))}=\frac{\H^2(C)}{\H^1(\partial C)}
    \]
    holds for every $K\in \K_C$ if and only if $C$ is a $2$-dimensional Cheeger set. 
\end{corollary}

For dimension $n=2$, the analog of Theorem~\ref{Characterization theorem n=3} is satisfied in so far as all four claims are true. In fact, (iv) is satisfied, because every $1$-dimensional convex body is a segment, and in turn a Cheeger set. Then Corollary~\ref{Sufficient Cheeger} implies the analogues of (i), (ii) and (iii).

\section{Inequality \eqref{e: 1} under projection constraints}\label{s:5}

The proof of Theorem~\ref{Equal projec} is a fairly direct application of the linear refinement in Theorem~\ref{th: linear BM} of the Brunn-Minkowski inequality under the assumption of projection constraints. 

\begin{proof}[Proof of Theorem~\ref{Equal projec}]
Let $K,L\in \K^n$ be convex bodies, and let $u\in\sfe$ be such that the projections  $P_{u^\perp}(K)$ and $P_{u^\perp}(L)$ are $(n-1)$-dimensional homothetic convex bodies. It is not difficult to see that there are homothetic copies of $K$ and $L$ having coincident projections onto $u^\perp$. W.l.o.g., $P_{u^\perp}(K)=\mu P_{u^\perp}(L)$ for some $\mu>0$. Then,
\begin{equation}
\begin{array}{ccccc}
P_{u^\perp}(K)&=&(1-\lambda) (P_{u^\perp}(K)+P_{u^\perp}(L))&=&(1-\lambda)P_{u^\perp}(K+L),\\ 
P_{u^\perp}(L) &=& \lambda (P_{u^\perp}(K)+P_{u^\perp}(L)) &=&\lambda P_{u^\perp}(K+L)\,
\end{array}
\label{eq:pfE}
\end{equation}
for $0 < \lambda:=\frac{1}{1+\mu} < 1$. Thus, $\frac{1}{1-\lambda}K$ and $\frac{1}{\lambda}L$ belong to the same canal class in direction $u$, as $P_{u^\perp}\left(\frac{1}{1-\lambda}\,K\right)=P_{u^\perp}(K+L)=P_{u^\perp}\left(\frac{1}{\lambda}\,L\right)$. 

Theorem~\ref{th: linear BM} yields
\[
\begin{split}
\frac{\H^n(K+L)}{\H^{n-1}(P_{u^\perp}(K+L))} &=\frac{\H^n\left( (1-\lambda) \frac{1}{1-\lambda}K+ \lambda \frac{1}{\lambda}L\right)}{\H^{n-1}(P_{u^\perp}(K+L))} \\
&\ge (1-\lambda) \frac{\H^n\left(\frac{1}{1-\lambda}K\right)}{\H^{n-1}(P_{u^\perp}(K+L))}+\lambda \frac{\H^n\left( \frac{1}{\lambda}L\right)}{\H^{n-1}(P_{u^\perp}(K+L))}\\
&=(1-\lambda) \frac{\H^n\left(\frac{1}{1-\lambda}K\right)}{\H^{n-1}\left(\frac{1}{1-\lambda}P_{u^\perp}(K)\right)}+\lambda \frac{\H^n\left( \frac{1}{\lambda}L\right)}{\H^{n-1}\left(\frac{1}{\lambda}P_{u^\perp}(L)\right)}\\
&=\frac{\H^n(K)}{\H^{n-1}(P_{u^\perp}(K))}+ \frac{\H^n(L)}{\H^{n-1}(P_{u^\perp}(L))},
\end{split}
\]
where we also used \eqref{eq:pfE} in the second equality. This finishes the proof of inequality~\eqref{ineq equal projec}.

The remainder of Theorem~\ref{Equal projec} is given separately in the following proposition.
\end{proof}

We observe that, in the previous proof, we applied Theorem \ref{th: linear BM} using the stronger assumption of equal projections of (particular homothets of) $K$ and $L$ onto $u^\perp$, instead of equal volumes of the projections of $K$ and $L$ onto $u^\perp$. It is natural to wonder whether inequality~\eqref{ineq equal projec} may also hold under the milder assumption that $\H^{n-1}(P_{u^\perp} (K))=\H^{n-1}(P_{u^\perp} (L))$. The next proposition shows that we cannot weaken the assumption to equal volume of the projections.

\begin{proposition}\label{th: counterexample equal volume}
Let $n\geq 3$. There exist two convex bodies $K,L\in \K^n_n$, and $u\in \sfe$, such that $\H^{n-1}(P_{u^\perp}(K))=\H^{n-1}(P_{u^\perp}(L)) \ne 0$, but inequality~\eqref{ineq equal projec} does not hold.
\end{proposition}

Proposition~\ref{th: counterexample equal volume} is prepared by two lemmas.

\begin{lemma}\label{lem:Pf  counterexample equal volume 1}
For every $n \ge 3$, there is an $n$-dimensional body $K_n \in\K^n_n$ such that 
\begin{equation}\label{eq:pfL1 counterexample equal volume}
\frac{\H^n(K_n)}{\H^{n-1}(P_{e_1^\perp}(K_n))} > \frac{\H^{n-1}(P_{e_2^\perp}(K_n))}{\H^{n-2}(P_{e_2^\perp}(P_{e_1^\perp}(K_n)))}.
\end{equation}
\end{lemma}

\begin{proof}
We start with $n=3$. Let $K_3=\conv\{\underline{0},e_1,e_2,e_1+e_2,e_3\}$. It is not difficult to verify that $K_3=\conv\{\underline{0},e_1,e_2,e_3\}\cup \conv\{e_1,e_2,e_1+e_2,e_3\}$ is a union of two simplices with disjoint interiors, and that the projection of $K_3$ onto $e_1^\perp$ (resp. onto $e_2^\perp$) is the two-dimensional simplex $\conv\{\underline{0},e_2,e_3\}$ (resp. $\conv\{\underline{0},e_1,e_3\}$). Therefore, direct computation provides
\[
\begin{array}{ccccc}
\H^3(K_3)&=&\H^3(\conv\{\underline{0},e_1,e_2,e_3\})+\H^3(\conv\{e_1,e_2,e_1+e_2,e_3\}) &=&1/3,\\
\H^2(P_{e_1^\perp}(K_3))& =&\H^2(\conv\{\underline{0},e_2,e_3\})&=&1/2,\\
\H^2(P_{e_2^\perp}(K_3))& =&\H^2(\conv\{\underline{0},e_1,e_3\})&=&1/2,\\
\H^1(P_{e_2^\perp}(P_{e_1^\perp}(K_3)))& =&\H^1(\conv\{\underline{0},e_3\})&=&1.
\end{array}
\]
This yields \eqref{eq:pfL1 counterexample equal volume} for $n=3$.

For $n \ge 4$, it is enough to put $K_n=K_3 \times \{0\}^{n-3} + \{0\}^3 \times [0,1]^{n-3}$. Then \eqref{eq:pfL1 counterexample equal volume} also applies to $K_n$, because $\H^n(K_n)=\H^3(K_3)$, $\H^{n-1}(P_{e_1^\perp}(K_n))=\H^{2}(P_{e_1^\perp}(K_3))$, $\H^{n-1}(P_{e_2^\perp}(K_n))=\H^2(P_{e_2^\perp}(K_3))$, and $\H^{n-2}(P_{e_2^\perp}(P_{e_1^\perp}(K_n)))=\H^1(P_{e_2^\perp}(P_{e_1^\perp}(K_3)))$.
\end{proof}

\begin{lemma}\label{lem:Pf  counterexample equal volume 2}
Let $n\geq 3$, let $K \in \K^n_n$ be an $n$-dimensional convex body, and, for $h > 0$, let $L_h=\left[\underline{0},h^{n-2} e_2\right] + \left[\underline{0},\frac{1}{h} e_3\right] + \ldots + \left[\underline{0},\frac{1}{h} e_n\right]$. Then
\[
\lim_{h \to \infty} \frac{\H^n(K+L_h)}{\H^{n-1}(P_{e_1^\perp}(K+L_h))}=\frac{\H^{n-1}(P_{e_2^\perp}(K))}{\H^{n-2}(P_{e_2^\perp}(P_{e_1^\perp}(K)))}.
\]
\end{lemma}

\begin{proof}
We put $Q=[\underline{0},e_3]+\ldots+[\underline{0},e_n]$, so that $L_h=\frac{1}{h}Q+\left[\underline{0},h^{n-2}e_2\right]$.

Let $a=\min\{\langle x,e_2 \rangle:\ x \in K\}$ and $b=\max\{\langle x,e_2 \rangle:\ x \in K\}$. Then, 
for $h$ large enough, we have 
\[
\begin{split}
K+L_h & \left\{
\begin{array}{l}
\subset \left(P_{e_2^\perp}(K)+\frac{1}{h}Q\right)+\left[ae_2,(h^{n-2}+b)e_2\right],\\
\supset P_{e_2^\perp}(K)+\left[be_2,(h^{n-2}+a)e_2\right],
\end{array}\right.\\
P_{e_1^\perp}(K+L_h) & \left\{
\begin{array}{l}
\subset \left(P_{e_1^\perp}(P_{e_2^\perp}(K))+\frac{1}{h}Q\right)+\left[ae_2,(h^{n-2}+b)e_2\right],\\
\supset P_{e_1^\perp}(P_{e_2^\perp}(K))+\left[be_2,(h^{n-2}+a)e_2\right].
\end{array}\right.
\end{split}
\]
Thus, we have
\[
\begin{split}
\H^n(K+L_h) & \left\{
\begin{array}{l}
\le \H^{n-1}\left(P_{e_2^\perp}(K)+\frac{1}{h}Q\right) \left((h^{n-2}+b)-a\right),\\
\ge \H^{n-1}\left(P_{e_2^\perp}(K)\right)\left((h^{n-2}+a)-b\right),
\end{array}\right.\\
\H^{n-1}\left(P_{e_1^\perp}(K+L_h)\right) & \left\{
\begin{array}{l}
\le \H^{n-2}\left(P_{e_1^\perp}(P_{e_2^\perp}(K))+\frac{1}{h}Q\right)\left((h^{n-2}+b)-a\right),\\
\ge \H^{n-2}\left(P_{e_1^\perp}(P_{e_2^\perp}(K))\right)\left((h^{n-2}+a)-b\right),
\end{array}\right.
\end{split}
\]
for $h$ large enough. In turn, we have
\[
\frac{\H^n(K+L_h)}{\H^{n-1}\left(P_{e_1^\perp}(K+L_h)\right)} \left\{
\begin{array}{l}
\le \frac{\H^{n-1}\left(P_{e_2^\perp}(K)+\frac{1}{h}Q\right) \left(h^{n-2}+(b-a)\right)}{\H^{n-2}\left(P_{e_1^\perp}(P_{e_2^\perp}(K))\right)\left(h^{n-2}-(b-a)\right)},\\
\ge \frac{\H^{n-1}\left(P_{e_2^\perp}(K)\right)\left(h^{n-2}-(b-a)\right)}{\H^{n-2}\left(P_{e_1^\perp}(P_{e_2^\perp}(K))+\frac{1}{h}Q\right)\left(h^{n-2}+(b-a)\right)}.
\end{array}\right.
\]
Passing to the limit as $h \to \infty$ completes the proof.
\end{proof}

\begin{proof}[Proof of Proposition~\ref{th: counterexample equal volume}]
We take $u=e_1$ and $K=K_n$, as in Lemma~\ref{lem:Pf  counterexample equal volume 1}. W.l.o.g. we assume $\H^{n-1}(P_{e_1^\perp}(K_n))=1$. Then we use the convex body
$L_h$ from Lemma~\ref{lem:Pf  counterexample equal volume 2} to get
\[
\lim_{h \to \infty} \frac{\H^n(K_n+L_h)}{\H^{n-1}(P_{e_1^\perp}(K_n+L_h))}=\frac{\H^{n-1}(P_{e_2^\perp}(K_n))}{\H^{n-2}(P_{e_2^\perp}(P_{e_1^\perp}(K_n)))}
<\frac{\H^n(K_n)}{\H^{n-1}(P_{e_1^\perp}(K_n))}
.
\]
Note also that, by construction, $\H^{n-1}(P_{e_1^\perp}(L_h))=1$.

We observe that $L_h$ is $(n-1)$-dimensional. Since we require to find two full-dimensional convex bodies, let $\tilde{L}_h:=\frac{1}{h}[\underline{0},e_1]+L_h$, which is a full-dimensional convex body for every $h$, such that $\H^{n-1}(P_{e_1^\perp}(\tilde{L}_h))=\H^{n-1}(P_{e_1^\perp}(L_h))=1$.

Note that we have
\[
\lim_{h \to \infty} \frac{\H^n(K_n+\tilde{L}_h)}{\H^{n-1}(P_{e_1^\perp}(K_n+\tilde{L}_h))}=\lim_{h \to \infty} \frac{\H^n(K_n+L_h)}{\H^{n-1}(P_{e_1^\perp}(K_n+L_h))},
\]
which yields that \eqref{ineq equal projec} fails for large $h$.
\end{proof}

\end{document}